\documentclass[11pt]{amsart}%
\usepackage[usenames]{color}
\usepackage{amssymb}
\usepackage{graphicx}
\usepackage{amscd}
\usepackage[latin1]{inputenc}

\usepackage[colorlinks=true,linkcolor=webgreen, filecolor=webbrown,citecolor=webgreen]{hyperref}

\definecolor{webgreen}{rgb}{0,.5,0}
\definecolor{webbrown}{rgb}{.6,0,0}

\usepackage{color}
\usepackage{fullpage}
\usepackage{float}

\usepackage{graphics,amsmath,amssymb}
\usepackage{amsthm}
\usepackage{amsfonts}
\usepackage{latexsym}
\usepackage{epsf}

\newcommand{\seqnum}[1]{\href{http://oeis.org/#1}{\underline{#1}}}

\theoremstyle{plain}
\newtheorem{teor}{Theorem}
\newtheorem{coro}[teor]{Corollary}

\newtheorem{prop}[teor]{Proposition}

\theoremstyle{definition}
\newtheorem{defi}[teor]{Definition}
\newtheorem{ejem}[teor]{Example}

\newtheorem{caso}{Case}

\theoremstyle{remark}

\def\du{\mathop{\rm du}\nolimits}
\def\pdu{\mathop{\rm pdu}\nolimits}
\def\rdu{\mathop{\rm rdu}\nolimits}
\title{The set of $k$-units modulo $n$}
\author[J.H. Castillo]{John H. Castillo}
\address{John H. Castillo, Departamento de Matemáticas y Estadística, Universidad de Nariño}
\email{jhcastillo@udenar.edu.co}
\author[J. Caranguay]{Jhony Fernando Caranguay Mainguez}
\address{Jhony Fernando Caranguay Mainguez, Departamento de Matemáticas y Estadística, Universidad de Nariño}
\email{jfernandomainguez@gmail.com}

\begin{document}
\maketitle

\begin{abstract}
Let $R$ be a ring with identity, $\mathcal{U}(R)$ the group of units of $R$ and $k$ a positive integer. We say that $a\in \mathcal{U}(R)$ is $k$-unit if $a^k=1$. Particularly, if the ring $R$ is $\mathbb{Z}_n$, for a positive integer $n$, we will say that $a$ is a $k$-unit modulo $n$. We denote with $\mathcal{U}_k(n)$ the set of $k$-units modulo $n$. By $\du_k(n)$ we represent the number of $k$-units modulo $n$ and with $\rdu_k(n)=\frac{\phi(n)}{\du_k(n)}$ the ratio of $k$-units modulo $n$, where $\phi$ is the Euler phi function. Recently, S. K. Chebolu proved that the solutions of the equation $\rdu_2(n)=1$ are the divisors of $24$. The main result of this work, is that for a given $k$, we find the positive integers $n$ such that $\rdu_k(n)=1$. Finally, we give some connections of this equation with Carmichael's numbers and two of its generalizations: Kn\"odel numbers and generalized Carmichael numbers.
\end{abstract}

S. K. Chebolu \cite{Ch12} proved that in the ring $\mathbb{Z}_n$ the square of any unit is $1$ if and only if $n$ is a divisor of $24$. This property is known as the \textit{diagonal property} for the ring $\mathbb{Z}_n$. Later, K. Genzlinger and K. Lockridge  \cite{GL15} introduced the function $\du(R)$, which is the number of involutions in $R$ (that is the elements in $R$ such that $a^2=1$), and provided another proof to Chebolu's result about the diagonal property. The diagonal property also has been studied in other rings. For instance, S. K. Chebolu \cite{ch13} found that the polynomial ring $Z_n[x_1,x_2,\ldots,x_m]$ satisfies the diagonal property if and only if $n$ is a divisor of $12$ and S. K. Chebolu et al. \cite{ch15} also characterized the group algebras that verifies this property.

Let $R$ be a ring with identity and $\mathcal{U}(R)$ the group of units of $R$. The aim of this paper is to study the elements of a ring with the following property: for a given $k\in \mathbb{Z}^+$, we say that an element $a$ in $\mathcal{U}(R)$ is a $k$-unit if $a^k=1$. So, we ask for the number of this elements, and for that we extend the definitions of the functions given by  K. Genzlinger and K. Lockridge  \cite{GL15}, particularly $\du_k(R)$ will represent the number of $k$-units of $R$. Here we present a formula for this function when $\mathcal{U}(R)$ can be expressed as a finite direct product of finite cyclic groups and when $R=\mathbb{Z}_n$. Furthermore, we study the case when $R=\mathbb{Z}_n$ and each unit is a $k$-unit. Previously, as mentioned before, this problem has been considered when $k=2$ and more generally for fields and group algebras when $k$ is a prime number, see \cite{ch15}.

In the other hand, a well studied topic in number theory are the Carmichael's numbers, which in terms of the $k$-unit concept, are composite positive integers such that any unit is an $(n-1)$-unit. Here, we find some connections between the equation $\rdu_k(n)$ and the concepts of Kn\"odel and  generalized Carmichael numbers.

In the sequel, for $x$ an element of a group $G$, by $|x|$ we denote the  order of $x$. Besides, for a prime number $p$ and a positive integer $n$, the symbol $\nu_p(n)$ means the exponent of the greatest power of $p$ that divides $n$, $\gcd(a,b)$ denotes the  greatest common divisor of $a$ and $b$,  and $\phi$ is the Euler's totient function. If $A=\{a_1,a_2,\ldots,a_n\}$ and $f$ is a defined function on $A$, we write 
$$\prod_{a\in A}f(a)=f(a_1)\cdot f(a_2)\cdots f(a_n),$$
and when $A=\varnothing$, we assume that $\prod_{a\in A}f(a)=1$.

\section{Set of $k$-units of a ring}
In this section we give some definitions and get some preliminary results.
\begin{defi} Let $R$ be a ring with identity, $a\in R$ and $k\in \mathbb{Z}^+$. We say that $a$ is a $k$-unit of $R$ if $a^{k}=1$. We will denote with $\mathcal{U}_k(R)$ the set of $k$-units of $R$.
\end{defi}

When $R=\mathbb{Z}_n$, for a given $n\in \mathbb{Z}^+$, we will use the symbol $\mathcal{U}_k(n)$ to denote the set of $k$-units of $\mathbb{Z}_n$, and we will call it the set of $k$-units modulo $n$.

\begin{ejem} In $\mathbb{Z}_5$, when we square the elements of $\mathcal{U}(\mathbb{Z}_5)$ we have that
\[
1^2=1,\;2^2=4,\;3^2=4,\; 4^2=1.
\]
Then, the $2$-units modulo $5$ are $1$ and $4$, that is, $\mathcal{U}_2(5)=\left\{1, 4\right\}$.
\end{ejem}
We can verify that $\mathcal{U}_2(5)$ is a subgroup of $\mathcal{U}(\mathbb{Z}_5)$. Actually, this is a property for rings with an abelian unit group.

\begin{teor}\label{teor1}
Let $R$ be a ring with identity such that $\mathcal{U}(R)$ is an abelian group. Then $\mathcal{U}_k(R)$ is a subgroup of $\mathcal{U}(R)$.
\end{teor}

\begin{proof} It is sufficient to prove that if $a$, $b$ $\in \mathcal{U}_k(R)$, then $ab^{-1}\in \mathcal{U}_k(R)$. Indeed, if $a^{k} = 1$ and $b^k= 1$, then $(ab^{-1})^k=a^k(b^k)^{-1}= 1$.
\end{proof}

When $\mathcal{U}_k(R)$ is a finite set, we will denote with $\du_k(R)$ the number of $k$-units of $R$; that is, 
\begin{equation}\label{du_def}
\du_k(R)=|\mathcal{U}_k(R)|.
\end{equation} 
Specially, if $R=\mathbb{Z}_n$,  $\du_k(n)$ will represents the number of $k$-units modulo $n$.

Although, in our definitions $k$ can be any positive integer, actually we could restrict it to the set of divisors of $|\mathcal{U}(R)|$, of course when the latter is finite. In fact,  take $d=\gcd(k,|\mathcal{U}(R)|)$. If $x\in  \mathcal{U}_k(R)$, then  $x^k=1$, and therefore the order of $x$ divides $k$. Moreover, as $|x|$ is a divisor of  $|\mathcal{U}(R)|$, also divides $d$. Thus, $x^d=1$, and then $x\in  \mathcal{U}_d(R)$, which implies that $\mathcal{U}_k(R)\subseteq\mathcal{U}_d(R)$. Similarly, we can prove that if $x\in  \mathcal{U}_d(R)$, then $x\in  \mathcal{U}_k(R)$. So, we have proved that $\mathcal{U}_k(R)=\mathcal{U}_d(R)$, result that we summarize in the next proposition.

\begin{prop}\label{prop1}
Let $k$ be a positive integer and assume that $\mathcal{U}(R)$ is finite. Then
$$ \mathcal{U}_k(R)=\mathcal{U}_d(R), $$
where $d=\gcd(k,|\mathcal{U}(R)|)$.
\end{prop}
The following result is an special case of the previous one.
\begin{teor}\label{teor2}
If $\mathcal{U}(R)$ is a finite cyclic group, then $\du_k(R)=\gcd(k,|\mathcal{U}(R)|)$.
\end{teor}

\begin{proof} Let $x\in \mathcal{U}_k(R)$ and $g$ a generator of $\mathcal{U}(R)$. So, there exists an integer $0\leq l<|\mathcal{U}(R)|$ such that $x=g^l$.

Then, $x^k=g^{kl}=1$ if and only if $kl\equiv 0 \pmod{|\mathcal{U}(R)|}$. The last congruence has $\gcd(k,|\mathcal{U}(R)|)$ solutions modulo $|\mathcal{U}(R)|$, see \cite[Prop. 3.3.1]{R90}. Thus, $x$ takes $\gcd(k,|\mathcal{U}(R)|)$ values and, therefore, $\du_k(R)=\gcd(k,|\mathcal{U}(R)|)$.
\end{proof}

We can give another proof to the Theorem \ref{teor2} using the following property of the Euler's $\phi$ function, $\sum_{e|d} \phi(e)=d$, see \cite[Prop. 2.2.4]{R90}.

\begin{proof} Take $d=\gcd(k,|\mathcal{U}(R)|)$. Then $x\in \mathcal{U}_k(R)= \mathcal{U}_d(R)$ if and only if the order of $x$ divides  $d$. Thus, the number of $k$-units in $R$ is equal to the number of elements of $\mathcal{U}(R)$ such that its order is a divisor of $d$. So,
\begin{align*}
\du_k(R)&=\du_d(R)=|\{x\in \mathcal{U}(R): |x|\text{ divides }d\}|\\
&=\sum_{e|d} |\{x\in \mathcal{U}(R): |x|=e\}|.
\end{align*}
As $e$ is a divisor of $d$, then it is also a divisor of $|\mathcal{U}(R)|$. Therefore, the number of elements of order $e$ in $\mathcal{U}(R)$ is $\phi(e)$, see \cite[Thm.\ 4.4]{G09} and thus
$$\du_k(R)=\sum_{e|d} \phi(e)=d.$$
\end{proof}
Now we are interested in finding an expression to this function when the group $\mathcal{U}(R)$ is isomorphic to the direct product of finite cyclic groups. Here and subsequently $C_r$ denotes the cyclic group of order $r$.

In some occasions we will apply our definition of $k$-unit and the function $\du_k$, $\pdu_k$ and $\rdu_k$ for groups. Previously, it was unnecessary, because it might be ambiguous, for instance, $\du_k(\mathbb{Z}_n)$ could be understand as the quantity of $k$-units of a ring or a group.

\begin{teor} \label{teor3}
Let $R$ be a commutative ring with identity. If $\mathcal{U}(R)\cong C_{r_1}\times\cdots \times C_{r_s}$ for some positive integers $r_1,\ldots, r_s$, then
\begin{equation*}
\du_k(R)=\prod^{s}_{i=1}\gcd(k,r_i).
\end{equation*}

\end{teor}

\begin{proof} By the given isomorphism we get that $\du_k(R)=\du_k(C_{r_1}\times\cdots\times C_{r_s})$. Let $a$ be a $k$-unit of $R$ and $b$ its image under the isomorphism. Then $b$ is a $k$-unit of $C_{r_1}\times\cdots\times C_{r_s}$ if and only if each $i$-th component of $b$ is a $k$-unit in $C_{r_i}$. So, $\du_k(R)=\du_k(C_{r_1})\cdots \du_k(C_{r_s})$. In this way, from Theorem \ref{teor2}, we obtain that $\du_k(R)=\gcd(k,|C_{r_1}|)\cdots \gcd(k,|C_{r_s}|)$, and the result follows.
\end{proof}

If $\mathcal{U}(R)$ is finite, we define the proportion and ratio functions of $k$-units of $R$, $\pdu_k(R)$ and $\rdu_k(R)$, respectively as follows
\begin{equation}
\begin{split}	
\pdu_k(R)&=\frac{\du_k(R)}{|\mathcal{U}(R)|},\\
\rdu_k(R)&=\frac{|\mathcal{U}(R)|}{\du_k(R)}=\frac{1}{\pdu_k(R)}.
\end{split}
\label{ec_radio}
\end{equation}

When $\mathcal{U}(R)$ is an abelian group, Theorem \ref{teor1} and Lagrange's Theorem, see \cite[Thm.~7.1]{G09}, guarantee that $\du_k(R)$ divides $|\mathcal{U}(R)|$, and therefore $\rdu_k(R)\in \mathbb{Z}^+$. For the ring $\mathbb{Z}_n$, we will use $\pdu_k(n)$ and $\rdu_k(n)$ to denote the proportion and ratio functions of the $k$-units modulo $n$, respectively. Since $|\mathcal{U}(\mathbb{Z}_n)|=\phi(n)$, we have that 
$$\pdu_k(n)=\frac{\du_k(n)}{\phi(n)}$$ and 
$$\rdu_k(n)=\frac{\phi(n)}{\du_k(n)}.$$

\begin{ejem} Previously, we got that $\mathcal{U}_2(5)=\left\{ 1, 4 \right\}$. Thus,
\begin{center}
$\du_2(5)=|\mathcal{U}_2(5)|=2$, $\pdu_2(5)=\dfrac{\du_2(5)}{|\mathcal{U}(\mathbb{Z}_5)|}=\dfrac{1}{2}$, and $\rdu_2(5)=\dfrac{1}{\pdu_k(5)}=2$.
\end{center}
\end{ejem}

\section{The group of $k$-units modulo $n$}

In this section we find an expression for $\du_k(n)$ from the prime factorization of $n$.

The following theorem shows that the functions given by \eqref{du_def} and \eqref{ec_radio} are multiplicatives when $R=\mathbb{Z}_n$, which implies that the task is reduced to calculate $\du_k$ for powers of primes.

\begin{teor}\label{teor_multiplica}
The functions $\du_k$, $\pdu_k$ and $\rdu_k$ defined on $\mathbb{Z}_n$ are multiplicatives.
\end{teor}

\begin{proof} We will demonstrate that if $s$ and $t$ are relatively primes positives integers, then $\du_k(st)=\du_k(s)\du_k(t)$.

By the Chinese Remainder Theorem, see \cite[Thm. 1', p.~35]{R90}, we have that $\mathbb{Z}_{st}\cong \mathbb{Z}_s\times \mathbb{Z}_t$, so the number of $k$-units modulo $n$ is equal to the number of $k$-units in  $\mathbb{Z}_s\times \mathbb{Z}_t$.

Let $(x, y)\in \mathbb{Z}_s\times \mathbb{Z}_t$ be a $k$-unit. Then $(x, y)^k=(x^k, y^k)=(1, 1)$ if and only if $x^k=1$ in $\mathbb{Z}_s$ and $y^k=1$ in $\mathbb{Z}_t$. Thus, $\left(x, y\right)$ is a $k$-unit of $\mathbb{Z}_s\times \mathbb{Z}_t$ if and only if $x$ is a $k$-unit modulo $s$ and $y$ is a $k$-unit modulo $t$. Therefore, $\du_k(st)=\du_k(\mathbb{Z}_s\times \mathbb{Z}_t)=\du_k(s)\du_k(t)$.

Since $\du_k$ and $\phi$ are multiplicatives, we have that
\begin{align*}
\pdu_k(st)&=\frac{\du_k(st)}{\phi(st)}=\left(\frac{\du_k(s)}{\phi(s)}\right)\left(\frac{\du_k(t)}{\phi(t)}\right)=\pdu_k(s)\pdu_k(t), \text{ and}\\
\rdu_k(st)&=\frac{1}{\pdu_k(st)}=\left(\frac{1}{\pdu_k(s)}\right)\left(\frac{1}{\pdu_k(t)}\right)=\rdu_k(s)\rdu_k(t).
\end{align*}

\end{proof}

By the last theorem and with the aim of finding an expression to $\du_k(n)$ using the prime factorization of $n$, we will consider when $n$ is a prime power. In order to apply the Theorem \ref{teor3}, we recall the following result, see \cite[p.\ 160]{G09}, which expresses $\mathcal{U}(\mathbb{Z}_{p^\alpha})$ as an external direct product of cyclic subgroups, where $p$ is a prime number and $\alpha$ is a positive integer.

\begin{prop} \label{prop2}
Let $p$ be a prime number and $\alpha$ a positive integer. Then
\[
\mathcal{U}(\mathbb{Z}_{p^\alpha})\cong
\begin{cases}
C_1,  & \text{if $p^\alpha=2^1$;}\\
C_2, & \text{if $p^\alpha=2^2$;}\\
C_2\times C_{2^{\alpha-2}}, & \text{if $p=2$ and $\alpha \geq 3$;}\\
C_{\phi(p^\alpha)}, & \text{if $p$ is odd.}
\end{cases}
\]
\end{prop}

\begin{teor}\label{alpha3}
If $\alpha$ is a positive integer greater than or equal to $3$, then
\[
\du_k(2^{\alpha})=
\begin{cases}
1,  & \text{if $k$ is odd;}\\
2\gcd(k,2^{\alpha-2}), & \text{if $k$ is even.}
\end{cases}
\]
\end{teor}

\begin{proof} By Proposition \ref{prop2}, we have that $\mathcal{U}(\mathbb{Z}_{2^\alpha})\cong C_2\times C_{2^{\alpha-2}}$. Applying Theorem \ref{teor3}, we get that $\du_k(2^\alpha)=\gcd(k, 2)\gcd(k, 2^{\alpha-2})$.
\end{proof}

From theorems \ref{teor2}, \ref{teor_multiplica} and \ref{alpha3} we obtain the following result.

\begin{teor}\label{du}
Assume that $n=2^{\alpha}m$, where $\alpha$ is a non-negative integer and $m$ is an odd positive integer. If $m=\prod_{i=1}^{r}p_i^{r_i}$ is the prime factorization of $m$, then
\begin{equation*}
\du_k(n)=\begin{cases}
\prod_{i=1}^r\gcd(k,\phi(p_i^{r_i})) & \text{if $k$ is odd or $\alpha=1$;}\\
2\prod_{i=1}^r\gcd(k,\phi(p_i^{r_i})) & \text{if $k$ is even and $\alpha=2$;}\\
2\gcd(k,2^{\alpha-2})\prod_{i=1}^r\gcd(k,\phi(p_i^{r_i})) & \text{if $k$ is even and $\alpha \geq 3$.}
\end{cases}
\end{equation*}

\end{teor}

\section{On the equation $\rdu_k(n)=1$}

Recently, S. K. Chebolu \cite{Ch12} studied the positive integers $n$ such that each $u\in \mathcal{U}(\mathbb{Z}_n)$ satisfies that $u^2=1$, which is equivalent to study the $n$'s that verify the equation $\rdu_2(n)=1$. In this section,  for a given positive integer $k$ we characterize the solutions of the equation $\rdu_k(n)=1$, which clearly is an extension of the work made by S. K. Chebolu. As a consequence, at the end of the section we obtained a new proof of the principal result of S. K. Chebolu \cite{Ch12}.

In general, we say that a ring satisfies the equation $\rdu_k(R)=1$, if its unit group is equal to its $k$-units set; it means that, $a^k=1$ for each $a\in \mathcal{U}(R)$. 

From Theorem \ref{teor2}, we get the result below.

\begin{teor} \label{t9} Let $R$ be a ring such that $\mathcal{U}(R)$ is a cyclic finite group. Then $\rdu_k(R)=1$ if and only if $|\mathcal{U}(R)|$ divides  $k$.
\end{teor}

The following theorem gives a condition to a ring $R$ to satisfy the equation $\rdu_k(R)=1$ when $\mathcal{U}(R)$ is isomorphic to a finite external direct product of finite cyclic groups. 

\begin{teor} \label{teor10}Let $R$ be a commutative ring with identity such that $\mathcal{U}(R)\cong C_{r_1}\times\cdots\times C_{r_s}$ for some positive integers $r_1,\ldots,r_s$. Then, $\rdu_k(R)=1$ if and only if $r_i$ divides $k$.
\end{teor}

\begin{proof} By the hypothesis $|\mathcal{U}(R)|=\prod^{s}_{i=1}r_i$ and, from Theorem \ref{teor3}, we have that $\du_k(R)=\prod^{s}_{i=1}\gcd(k,r_i)$. Thus, $\rdu_k(R)=1$ if and only if $\prod^{s}_{i=1}\gcd(k,r_i)=\prod^{s}_{i=1}r_i$.

Since that $\gcd(k,r_i)\leq r_i$, then $\gcd(k,r_i)=r_i$, which happens only when each $r_i$ divides $k$.
\end{proof}

In the following theorem, for a given positive integer $n$, we give necessary and sufficient conditions on $k$ such that $\rdu_k(n)=1$.

\begin{teor} \label{teor11} Assume that $n=2^{\alpha}m$, where $\alpha$ is a non-negative integer and $m$ is  an odd positive integer. Then, $\rdu_k(n)=1$ if and only if for each prime divisor $p$ of $m$ we have that $\phi(p^{\nu_p(m)})$ divides $k$  and one of the following conditions is satisfied 
\begin{enumerate}
	\item $\alpha\in \{0,1\}$,
	\item $\alpha=2$ and $k$ is even,
	\item $3\leq \alpha \leq\nu_2(k)+2$.
\end{enumerate}
Particularly, if $k$ is odd then either $n=1$ or $n=2$.
\end{teor}

\begin{proof}
The proof is straightforward from the fact that $\rdu_k$ is multiplicative. Indeed, if the prime factorization of $m$ is 
$$m=\prod_{\substack{p|m \\ p \text{ prime}}}p^{\nu_{p}(m)},$$ we have that $\rdu_k(n)=1$ if and only if
$\rdu_k(2^{\alpha})=1$ and $\rdu_k(p^{\nu_{p}(m)})=1$ for each prime divisor $p$ of $m$.

Therefore, from Theorem \ref{t9}, we obtain that $\phi(p^{\nu_p(m)})$ is a divisor of $k$. Besides, $\phi(2^{\alpha})$ divides $k$, when $\alpha\leq 2$, which demonstrates 1 and 2. On the other hand, if $\alpha\geq 3$, as $\du_k(2^{\alpha})=2^{\alpha-1}$,  Theorem \ref{alpha3} implies that $\alpha-2\leq \nu_2(k)$.

Conversely, we can prove that if $\alpha$ satisfies the given conditions, then $\rdu_k(2^{\alpha})=1$.

Since $\phi$ takes even values (except in $1$ or $2$), from the facts we have proved previously, then when $k$ is odd, we get that either $n=1$ or $n=2$.
\end{proof}

The above result allow us to conclude that the study of the equation $\rdu_k(n)=1$, is relevant when $k$ is even. We can now formulate our main result.

\begin{teor}\label{teor12} Assume that $k=2^{\beta}M$, with $\beta > 0$ and $M$ is an odd positive integer. Then $\rdu_k(n)=1$ if and only if $n$ is a divisor of  $$2^{\beta+2}\prod_{p \in \mathcal{A}} p \prod_{ q\in \mathcal{B}} q^{\nu_{q}(M)+1},$$
where 
$$\mathcal{A}:=\left\{p: p \text{ is prime, } p\nmid M \text{ y } p=2^{l}d+1, \text{ with } 0<l\leq \beta \text{ and } d|M\right\},$$
and
$$\mathcal{B}:=\left\{q: q \text{ is prime, } q|M \text{ and } q=2^ld+1, \text{ with } 0<l\leq \beta \text{ and } d|M\right\}.$$
\end{teor} 

\begin{proof}  Suppose that $n=2^{\alpha}m$, with $m$ an odd integer and $\alpha$ a non-negative integer. First of all, since $k$ is even Theorem \ref{teor11} guarantees the inequality $0\leq \alpha \leq \beta+2$.

Additionally, Theorem \ref{teor11} implies that $\phi(t^{\nu_t(m)})|k$ for each prime divisor $t$ of $m$; that is
\begin{equation}
t^{\nu_t(m)-1}(t-1)|k.
\label{eqt}
\end{equation}
We will study the expression \eqref{eqt} in two cases, depending on whether or not $t$ divides $M$.

\begin{caso}Assume that $t$ does not divide $M$. Then, \eqref{eqt} implies that $\nu_t(m)=1$ and $t-1|k$. Thus, 
$$t-1=2^ld,$$
with $0<l\leq \beta$ and $d$ is a divisor of $M$. The primes $t$ that verify the last conditions are joined in the set $\mathcal{A}$ defined at the statement of the theorem.
\end{caso}

\begin{caso} Suppose that $t$ divides $M$. Again, from \eqref{eqt} we obtain that $\nu_t(m)-1\leq \nu_t(M)$ and  $t-1$ is a divisor of
$$\dfrac{k}{t^{\nu_t(M)}},$$
this means $$t-1=2^{l}d,$$
where $0<l\leq \beta$ and $d$ is a divisor of $M$. These primes are the elements of the set $\mathcal{B}$ of the theorem.
\end{caso}
Therefore, $n$ is a solution of the equation $\rdu_k(n)=1$, if it has the form
$$n=2^{\alpha}\prod_{p \in \mathcal{A}} p^{r(p)} \prod_{ q\in \mathcal{B}} q^{s(q)},$$
with $0\leq \alpha\leq \beta+2$, $r(p)\in \{0,1\}$ for each $p\in \mathcal{A}$ and $0\leq s(q) \leq \nu_{q}(M)+1$ for each $q\in \mathcal{B}$. 
\end{proof}
An immediate consequence of the above theorem is the following result.
\begin{coro}
 Assume that $k=2^{\beta}M$, with $\beta > 0$ and $M$ an odd positive integer. Then, the number of solutions of $\rdu_k(n)=1$
is given by
$$(\beta+3)2^{|\mathcal{A}|}\prod_{q\in \mathcal{B}}(\nu_{q}(M)+2),$$
where $\mathcal{A}$ and $\mathcal{B}$ are as in the previous theorem.
\end{coro}

Theorem \ref{teor12} allow us to obtain some well known results as we will see in the next section; specially here we give another proof of the principal result of S.~	K. Chebolu about the divisors of $24$, see \cite[Thm.\ 1.1]{Ch12}.
\begin{coro}\label{demo24} Let $n$ be a positive integer. Then, $n$ has the diagonal property if and only if $n$ is a divisor of $24$.
\end{coro}

\begin{proof}
Suppose that $n$ has the diagonal property, that is $\rdu_2(n)=1$. It is sufficient to find the sets $\mathcal{A}$ and $\mathcal{B}$ of the statement of Theorem \ref{teor12}. In fact, it is easy to check that $\mathcal{A}=\{3=2^1\times 1+1\}$ and $\mathcal{B}=\varnothing$. Therefore, the solutions of the given equation are the divisors of $2^{1+2}\times 3=24$. 

For the reciprocal, it is enough to verify that $\rdu_2(n)=1$ when $n$ is a divisor of $24$.
\end{proof}

In the sequel, we give some examples.

\begin{ejem} Take $k=10=2\times 5$. Then  
$$\mathcal{A}=\{3=2^1\times 1+1,11=2^1\times 5+1\} \text{ and } \mathcal{B}=\varnothing.$$
Thus, the roots of $\rdu_{10}(n)=1$ are the divisors of $2^{1+2}\times 3\times 11=264$.
\end{ejem}

In the proof of the Corollary \ref{demo24} and in the above example  $\mathcal{B}$ is empty; however, this not happen always as we can see in the next example.

\begin{ejem} Consider $k=252=2^2\times 3^2\times 7$. Then, 
\begin{align*}
\mathcal{A}&=\{5,13,19,29,37,43, 127\} \text{ and}\\
\mathcal{B}&=\{3,7\},
\end{align*}
where, for instance, $3=2^1\times 1+1$ and $7=2^1\times 3+1$. Therefore, the solutions of $\rdu_{252}(n)=1$ are the divisors of 
$$2^{4}\times 5\times 13\times 19\times 29\times 37\times 43\times 127\times 3^3\times 7^2=153185861359440,$$ 
and there are $5\times 2^7\times (3+1)\times (2+1)=7680$ solutions.
\end{ejem}

\section{Consequences and further work}

Finally, in this section, we present how the results demonstrated previously serve to obtain some well known results about  Carmichael numbers \seqnum{A002997}, and to stablish some connections with two of its generalizations: Kn\"odel numbers \seqnum{A033553}, \seqnum{A050990}, \seqnum{A050993} and generalized Carmichale numbers \seqnum{A014117}, see \cite{CP05,HH99,M62,OEIS}.
\subsection{Carmichael numbers}

Let $a$ be a positive integer. We say that a composite number $n$ is a pseudoprime base $a$ if 
\begin{equation}\label{eqFLT}
a^{n-1}\equiv 1 \pmod n.
\end{equation}
When we do not know whether $n$ is composite or prime, but satisfies the congruence \eqref{eqFLT}, we say that $n$ is a probable prime base $a$.

The next theorem gives the number of bases $a$ in $\mathcal{U}(\mathbb{Z}_n)$ such that $n$ is a probable base $a$, see \cite[p. \ 165]{CP05}

\begin{teor}Let $n$ be an odd positive integer. Then the number of bases $a$ such that $n$ is a probable prime base $a$ is $$B_{pp}(n)= \prod_{p|n}\gcd(n - 1, p-1).$$
\end{teor}

\begin{proof} 
Let $n=\prod_{p|n}p^{\nu_p(n)}$ be the prime factorization of $n$. First, we observe that $B_{pp}(n)=\du_{n-1}(n)$. Then, from Theorem \ref{du}, we have that
$$B_{pp}(n)=\du_{n-1}(n)=\prod_{p|n}\gcd(n-1,\phi(p^{\nu_p(n)}))=\prod_{p|n}\gcd(n-1,p^{\nu_p(n)-1}(p-1)).$$
Since $\gcd(n-1, p)=1$, we obtain that 
$$B_{pp}(n)=\prod_{p|n}\gcd(n-1, p-1).$$
\end{proof}

\begin{defi} Let $n$ be an odd composite integer. We say that $n$ is a Carmichael number if $a^{n-1}\equiv 1 \pmod n$ for each positive integer $a$ relatively prime to $n$.
\end{defi}
Actually, $n$ is a Carmichael number if it is composite and $\rdu_{n-1}(n)=1$. This allow us, to use the previous theorems to prove some known results about Carmichael numbers.

\begin{teor}\label{teor14} Let $n$ be an odd and composite integer and $k$ relatively prime to $n$. Then, $\rdu_k(n)=1$ if and only if $n$ is squarefree and $p-1$ divides $k$ for each prime divisor $p$ of $n$.
\end{teor}

\begin{proof} Assume that $n=\prod_{p|n}p^{\nu_p(n)}$ is the prime factorization of $n$. By Theorem \ref{teor12}, we have that $\rdu_k(n)=1$ if and only if $\phi(p^{\nu_p(n)})=p^{\nu_p(n)-1}(p-1)$ divides $k$, which only happens when $\nu_p(n)=1$ and $p-1$ divides $k$.
\end{proof}

The last result is a generalization of the Korselt criterion, see \cite[Thm.\ 3.4.6]{CP05}.

\begin{coro}[Korselt criterion] Suppose that $n$ is an odd and composite integer. Then, $n$ is a Carmichael number if and only if $n$ is squarefree and for each prime $p$ dividing $n$ we have $p-1$ divides $n-1$.
\end{coro}

\begin{prop} Any Carmichael number has at least three prime factors.
\end{prop}

\begin{proof} Suppose that $n=pq$, with $p$ and $q$ different primes. As $\rdu_{pq-1}(n)=1$, then $\rdu_{pq-1}(pq)=\rdu_{pq-1}(p)\rdu_{pq-1}(q)=1$.

This implies $\rdu_{pq-1}(p)=1$ and $\rdu_{pq-1}(q)=1$. Now, since
\begin{align*}
\rdu_{pq-1}(p)&=\frac{\phi(p)}{\gcd(pq-1, \phi(p))}\\
&=\frac{p-1}{\gcd(pq-1, p-1)},
\end{align*}
then $\gcd(pq-1,p-1)=p-1$. Similarly, we can prove that $\gcd(pq-1,q-1)=q-1$.
Furthermore, we have that $\gcd(pq-1, p-1)=\gcd(pq-1, q-1)$ because  \linebreak $pq-1=(p-1)(q-1)+(p-1)+(q-1)$; that is
$p-1=q-1$, which is a contradiction.
\end{proof}

A. Makowski \cite{M62}  gave an extension to the concept of Carmichael number, named Kn\"odel numbers, in honour of the Austrian mathematician W. Kn\"odel \cite[p.\ 125]{R96}. 

\begin{defi} For $i\geq 1$, let $\mathcal{K}_i$ be the set of all the composite integers $n>i$ such that $a^{n-i}\equiv 1 \pmod{n}$  for any positive integer $a$ relatively prime to $n$. We call $\mathcal{K}_i$ the $i$-Kn\"odel set and its elements the $i$-Kn\"odel numbers.
\end{defi}
It is clear that, $\mathcal{K}_1$ is the set Carmichael numbers. Similarly, as we did with the Carmichael numbers, we can give an interpretation of the Kn\"odel sets from the concepts studied in this article; in fact, it is easy to stablish that  $n\in \mathcal{K}_i$ if and only if  $n$ is composite and $\rdu_{n-i}(n)=1$. 

For a fixed $k\in \mathbb{Z}^+$, using Theorem \ref{teor12}, we demonstrated that the set of solutions of the equation $\rdu_k(n)=1$ is finite. Although, this is not necessarily always true when $k$ depends on $n$; for instance $\mathcal{K}_i$ is infinite, see \cite{AGP94,M62}. 

 L. Halbeisen and N. Hungerb\"uhler \cite{HH99} proposed another generalization of the Carmichael number concept.

\begin{defi}
Fix an integer $k$, and let be
$$C_k=\{n\in \mathbb{N}: \min\{n,n+k\}>1 \text{ and } a^{n+k}\equiv a \pmod{n} \text{ for all } a\in \mathbb{N}\}.$$
\end{defi} 

L. Halbeisen and N. Hungerbühler proved that $C_1=\{2,6,42,1806\}$ and $C_k$ is infinite if $1-k>1$ is squarefree. We can stablish that for $k\in \mathbb{Z}$
\begin{equation}
C_k\subset \{n\in \mathbb{Z}^+: \rdu_{n+k-1}(n)=1\}.
\label{eq:conten}
\end{equation}

Therefore, the set $C_1$ joint with $1$ are the solutions of $\rdu_{n}(n)=1$. Furthermore, expression \eqref{eq:conten} shows that when $1-k>1$ is squarefree, there are infinitely many solutions to the equation $\rdu_{n+k-1}(n)=1$.

In the sequel, we pose certain questions regarding the number of solutions of the equation $\rdu_k(n)=1$, when $k$ depends on $n$.

\begin{itemize}

\item Are there infinitely many $n\in \mathbb{Z}^+$ shuch that $\rdu_{n+1}(n)=1$?

From Theorem \ref{teor14}, this question is equivalent to ask for the infinitude of positive squarefree integers $n$ such that $p-1$ divides $n+1$ for each prime divisor $p$ of $n$. 
 
Let $i\in \mathbb{N}$ and $a, b\in \mathbb{Z}$ 

\item Are there infinitely many $n\in \mathbb{Z}^+$ such that $\rdu_{n+i}(n)=1$? 
\item Are there infinite $n\in \mathbb{Z}^+$ such that $\rdu_{an+b}(n)=1$? 

When $b=0$, from Theorem \ref{teor11}, it is enough to take $n$ as a power of $2$.

\item In general, are there infinitely many $n\in \mathbb{Z}^+$ such that $\rdu_{f(n)}(n)=1$, for a polynomial $f(x)$ with integer coefficients? If the answer is negative, for which polynomials $f(x)$ the equation $\rdu_{f(n)}(n)=1$ has infinite many solutions and for which has finite many solutions?

\end{itemize}
\section{Acknowledgements}
The authors are members of the research group: ``Algebra, Teor\'ia de N\'umeros
y Aplicaciones, ERM". J.~H. Castillo was partially supported by the Vicerrector\'ia de Investigaciones Postgrados y Relaciones Internacionales  at Universidad de Nari\~no. The authors also were partially supported by COLCIENCIAS under the research project ``Aplicaciones a la teor\'ia de la informaci\'on y comunicaci\'on de los conjuntos de Sidon y sus generalizaciones (110371250560)''. 

We are grateful to Professor Gilberto Garc\'ia-Pulgar\'in for his suggestions that help to improve this article.

\bigskip
\hrule
\bigskip

\noindent 2010 {\it Mathematics Subject Classification}: Primary 
11A05, 11A07, 11A15, 16U60.

\noindent \emph{Keywords: } Diagonal property, diagonal unit, unit set of a ring, $k$-unit, Carmichael number, Kn\"odel number and Carmichael generalized number.

\bigskip
\hrule
\bigskip

\noindent (Concerned with sequences \seqnum{A033553}, \seqnum{A050990}, \seqnum{A050993} and \seqnum{A014117}.)

\bigskip
\hrule


\begin{thebibliography}{10}

\bibitem{AGP94}
W.~R. Alford, A.~Granville, and C.~Pomerance, There are infinitely many
  {C}armichael numbers, {\em Ann. of Math. (2)} {\bf 139}(3) (1994), 703--722.

\bibitem{Ch12}
S.~K. Chebolu, What is special about the divisors of 24?, {\em Math. Mag.} {\bf
  85}(5) (2012), 366--372.

\bibitem{ch15}
S.~K. Chebolu, K.~Lockridge, and G.~Yamskulna, Characterizations of {M}ersenne
  and 2-rooted primes, {\em Finite Fields Appl.} {\bf 35} (2015), 330--351.

\bibitem{ch13}
S.~K. Chebolu and M.~Mayers, What is special about the divisors of 12?, {\em
  Mathematics Magazine} {\bf 86}(2) (2013), 143--146.

\bibitem{CP05}
R.~Crandall and C.~Pomerance, {\em Prime numbers. {A} computational
  perspective}, Springer, New York, second edition, 2005.

\bibitem{G09}
J.~Gallian, {\em Contemporary Abstract Algebra}, Cengage Learning, Belmont, 7th
  edition.

\bibitem{GL15}
K.~Genzlinger and K.~Lockridge, Sophie {G}ermain primes and involutions of
  {$\Bbb{Z}_n^\times$}, {\em Involve} {\bf 8}(4) (2015), 653--663.

\bibitem{HH99}
L.~Halbeisen and N.~Hungerb\"uhler, On generalized {C}armichael numbers, {\em
  Hardy-Ramanujan J.} {\bf 22} (1999), 8--22.

\bibitem{R90}
K.~Ireland and M.~Rosen, {\em A classical introduction to modern number
  theory}, Vol.~84 of {\em Graduate Texts in Mathematics}, Springer-Verlag, New
  York, second edition, 1990.

\bibitem{M62}
A.~Makowski, Generalization of {M}orrow's {$D$} numbers, {\em Simon Stevin}
  {\bf 36} (1962/1963), 71.

\bibitem{R96}
P.~Ribenboim, {\em The new book of prime number records}, Springer-Verlag, New
  York, 1996.

\bibitem{OEIS}
N.~J.~A. Sloane, The on-line encyclopedia of integer sequences, June published
  electronically at http://oeis.org.

\end{thebibliography}
\end{document}